\documentclass[a4paper,12pt]{amsart}
\usepackage{amssymb}
\usepackage[latin1]{inputenc}
\usepackage[T1]{fontenc}
\usepackage[top=3.5cm, textheight=46\baselineskip]{geometry}
\usepackage[colorlinks,linkcolor=blue]{hyperref}
\def\zx{{\zeta_x}}
\def\zy{{\zeta_y}}

\def\m{{\mathfrak m}} \def\P{{\mathbb P}}
\def\C{{\mathbb C}} 
\def\Z{{\mathbb Z}} \def\cY{{\mathcal Y}}
\def\cX{{\mathcal X}} \def\cU{{\mathcal U}}
\def\cJ{{\mathcal J}} \def\cT{{\mathcal T}} 
\def\wt#1{\widetilde{#1}}
\def\wl#1{\overline{#1}} \def\wh#1{\widehat{#1}}
\def\la{\lambda} \def\al{\alpha}
\def\ep{\varepsilon} \def\vp{\varphi}
\def\sier#1{{\mathcal O}_{#1}}
\def\wlsier#1{\wl{\mathcal O}_{#1}}
\def\cO{{\mathcal O}}
\def\mapleft#1{\mathrel{%
\smash{\mathop{\longleftarrow}\limits^{#1}}}}
\def\mapright#1{\mathrel{%
\smash{\mathop{\longrightarrow}\limits^{#1}}}}
\def\mapdown#1{\Big\downarrow
 \rlap{$\vcenter{\hbox{$\scriptstyle #1$}}$}}
\def\cdmatrix#1{\def\normalbaselines{\baselineskip20pt
 \lineskip3pt \lineskiplimit3pt }\begin{matrix}#1\end{matrix}}
\def\lra{\longrightarrow}
\DeclareMathOperator{\Ker}{Ker}
\DeclareMathOperator{\Coker}{Coker}
\DeclareMathOperator{\codim}{codim}
\DeclareMathOperator{\rk}{rank}
\def\half#1#2{{\textstyle\frac{#1}{#2}}}

\newtheorem{lemma}{Lemma}
\newtheorem{proposition}[lemma]{Proposition}
\newtheorem{theorem}[lemma]{Theorem}
\newtheorem*{HSP}{The hyperplane section principle}
\newtheorem{corollary}[lemma]{Corollary}
\theoremstyle{definition}
\newtheorem{example}{Example}
\newtheorem{definition}{Definition}
\newtheorem{remark}{Remark}

\begin{document}
\title[Deforming nonnormal  singularities]%
{Deforming nonnormal isolated surface singularities
and constructing 3-folds with $\P^1$ as exceptional set}
\author{Jan Stevens}
\email{stevens@chalmers.se}
\address{Matematik,
G\"oteborgs universitet and Chalmers tekniska h\"ogskola,
SE 412 96 G\"oteborg, Sweden}

\begin{abstract}
Normally one assumes isolated surface singularities to be normal.
The purpose of this paper is to show that it can be useful to look
at nonnormal singularities. By deforming them interesting normal
singularities can be constructed, such as isolated, non Cohen-Macaulay
threefold singularities. They arise by a small contraction of a smooth
rational curve, whose normal bundle has a sufficiently positive subbundle.
We study such singularities from their nonnormal general hyperplane
section.

\end{abstract}
\dedicatory{To Gert-Martin Greuel on the occasion of his 
70th birthday.}

\maketitle

\section*{Introduction}
Suppose we are interested in a germ $(X,0)\subset (\C^N,0)$ 
of a complex space, which has some salient features. Then we 
would like to describe the singularity $X$ as explicit as possible.
This can be done by giving generators of the local ring $\sier X$, 
or by giving equations for $X\subset \C^N$. 
But in general it is too difficult
to do this directly. Instead we first replace the singularity
by a simpler one. To recover
the original one is then a deformation problem.
In a number of situations the simplification process leads to
nonnormal singlarities.

We formulate the most important simplification process 
a general principle.

\begin{HSP}
The {\rm({\em general\/})} hyperplane section of a singularity has a local
ring with the same structure as the original singularity, but
one embedding dimension lower, and which is much easier to describe. 
\end{HSP}

A nonnormal surface singularity occurs as  
general hyperplane of a normal three-dimensional isolated
singularity,  if this singularity is not Cohen-Macaulay.
Such singularities can occur as result of small contractions.
In higher dimensions 
a resolution (with normal crossings
exceptional divisor) is in general not the correct tool for
understanding the singularity. But it may
happen that a small resolution exists, 
meaning (in dimension three) that the 
exceptional set is only a curve. The simplest case is that
the curve is a smooth rational curve. Nevertheless, such a 
singularity can be quite complicated, as it need not be 
Cohen-Macaulay. This happens if the normal bundle of the
curve is $\cO(a)\oplus\cO(b)$ with $a>1$. One has always
that $2a+b<0$ \cite{And2, Nak}, and Ando has given examples
of the extremal case $(a,b)=(n,-2n-1)$ by exhibiting 
transition functions. We study the contraction of such a curve
using the hyperplane section principle. 

The first example of a manifold containing an 
exceptional $\P^1$ with normal bundle with positive subbundle,
namely $\cO(1)\oplus\cO(-3)$, was given by Laufer \cite{Lau2}.
Pinkham gave a construction 
as total  space of a 1-parameter smoothing of a partial
resolution of a rational double point \cite{Pin2}.
We consider smoothings of partial resolutions of
non-rational singularities. In this case the total space does blow
down, but not to a smoothing of the original singularity. Instead, 
the special fibre is a nonnormal surface singularity.
We retrieve Ando's examples using the canonical model  of the hypersurface
singularity $z^2=f_{2n+1}(x,y)$. We also give new examples. 

Using \textsc{Singular} \cite{GLS} it is possible to give explicit 
equations for some cases. We do not compute deformations of the
canonical model, but deformations of the nonnormal surface
singularity $X$. We study in detail the simplest case, where
$X$ differs not too much from its
normalisation $\wt X$, meaning that  $\delta(X)= \dim \sier{\wt X}/\sier X=1$.
We take the equation of $\wt X$ of the form $z^2=f(x,y)$. It turns 
out that it is in fact possible to give general formulas.

A 1-parameter deformation of a  resolution of a normal surface
singularity blows down to a deformation of the singularity 
if and only if the geometric genus is constant. Otherwise the 
special fibre is nonnormal. 
Given a 2-dimensional hypersurface singularity, the general singularity
with the same resolution graph is not Gorenstein, and not quasi-homogeneous
in the case that hypersurface is quasi-homogenous.
Again, deforming a nonnormal quasi-homogeneous surface
singularity gives a method to find equations for surface 
singularities with a given star-shaped graph. We give an example
using the same general formulas as for small
contractions (Example \ref{ex_endrie}).

The structure of this paper is as follows.  In the first section
we discuss invariants for 
nonnormal surface singularities. 
In the next section we compute deformations
for a nonnormal model of a surface singularity of multiplicity two.
In section 3 we recall in detail the relation between
deformations of a (partial) resolution and of the singularity itself.
The final section treats $\P^1$
as exceptional curve, with explicit formulas based on the 
previous calculations.

\section{Invariants of nonnormal singularities}
\subsection{Normalisation}
\begin{definition}
A reduced ring $R$ is \textit{normal} if it is
integrally closed in its total ring of
fractions. For an arbitrary reduced ring $R$ its
\textit{normalisation} $\wl R$ is the
integral closure of $R$  in its total ring of
fractions.
A singularity $(X,0)$ (i.e., the germ of a complex space) is \textit{normal} if its local ring
$\sier{(X,0)}$ is normal. The \textit{normalisation} of a
reduced germ  $(X,0)$ is a multi-germ $(\wl X,\wl 0)$
with semi-local ring $\sier{(\wl X,\wl 0)}= \wlsier{(X,0)}$.
The normalisation map is $\nu\colon (\wl X,\wl 0) \to (X,0)$,
or in terms of rings
$\nu^*\colon \sier{(\wl X,\wl 0)}
\to \sier{(X,0)}$.
\end{definition}
We have the following  function-theoretic characterisation of
normality, see e.g. \cite[p.~143]{AS}. 
Let $\Sigma$ be the singular locus of a reduced complex space $X$
and set  $U=X\setminus \Sigma$, with $j\colon U\to X$
the inclusion map. Then $X$ is normal at $p\in X$ 
if and only if for arbitrary small neighbourhoods $V\ni p$  every 
bounded holomorphic function on $U\cap V$ has a holomorpic extension to 
$X\cap V$.
If $\codim \Sigma  \geq 2$, then $\wlsier X= j_*\sier U$.

\subsection{Cohen-Macaulay singularities}
For a two-dimensional isolated singularity normal is equivalent to 
Cohen-Macaulay, but in higher dimensions this is no longer
true. 
A local ring is \textit{Cohen-Macaulay} if there is a regular sequence
of length equal to the dimension of the ring.
A $d$-dimensional germ $(X,0)$ is  
Cohen-Macaulay, if its local ring is Cohen-Macaulay. An equivalent
condition is that there exists a finite
projection $\pi\colon (X,0)\to (\C^d,0)$ with fibres of constant multiplicity
(i.e., the map $\pi$ is flat), 
see e.g., \cite[Kap. III § 1]{AS}.
From both descriptions it follows directly that
a singularity is Cohen-Macaulay if and only if its general
hyperplane section is so. In particular,
a  general hyperplane section of a normal but
not Cohen-Macaulay isolated 3-fold singularity is not normal.

A cohomological characterisation, in terms of local cohomology,
of  
isolated Cohen-Macaulay singularities is that  $H^{q}_{\{0\}}(X,\sier
X)=0$ for $q<d$. Normality implies only the vanishing for $q< 2$.
The local cohomology can be computed from a resolution $\wt X \to X$, as
$H^q(\wt X,\sier {\wt X})\cong
H^{q+1}_{\{0\}}(X,\sier X)$ for $1\leq q \leq n-2$
\cite[Prop.~4.2]{Kar}.
If $\wt X$ is a good resolution of an isolated singularity with
exceptional divisor $E$, then 
the map $H^i(\sier {\wt X})\to 
H^i(\sier E)$ is surjective for all $i$ \cite[Lemma 2.14]{Stee}. 
This implies
that a 3-fold singularity is not Cohen-Macaulay if the 
exceptional divisor of a good resolution is an irregular surface 
$F$  (meaning that $q=h^1(\sier F)>0$).
The easiest example of such a singularity is the cone over an irregular
surface. 

\begin{example}
The equations of the known families of
smooth irregular surfaces in $\P^4$ are discussed in \cite[Sect.~4]{ADHPR}.
They admit a large symmetry group, the Heisenberg group. The lowest degree
case is that of elliptic quintic scrolls. Their homogeneous coordinate
ring has a minimal free resolution of type
\[
0 \longleftarrow \sier S \longleftarrow
\cO(-3)^5 \mapleft L\cO(-4)^5  \mapleft{x}
\cO(-5) \longleftarrow 0\;,
\]
where $L$ is a matrix
\[
\begin{pmatrix}
0         &  -s_1 x_4  &   -s_2 x_3 &  s_2x_2 & s_1x_1  \\
s_1 x_2  &        0   &   -s_1 x_0 & -s_2x_4 & s_2x_3  \\
s_2 x_4  &   s_1 x_3  &     0      & -s_1x_1 & -s_2x_0  \\
-s_2 x_1  &  s_2 x_0  &  s_1x_4    &0       & -s_1x_2  \\
-s_1 x_3  &  -s_2 x_2  & s_2x_1     & s_1x_0 &0    
\end{pmatrix}
\]
and $x$ is the vector $(x_0,x_1,x_2,x_3,x_4)^t$. The constants $(s_1:s_2)$
are homogeneous coordinates 
on the modular curve $X(5)\cong \P^1$. The $i$-th column
of the matrix $\bigwedge^4 L$ is divisible by $x_i$ and the equations of 
the scroll are the five resulting cubics, which can be obtained from
the following one by cyclic permutation of the indices:
\[
s_1^4 x_0x_2x_3-s_1^3s_2(x_1x_2^2+x_3^2x_4)-s_1^2s_2^2x_0^3
+s_1s_2^3(x_1^2x_3+x_2x_4^2)+s_2^4x_0x_1x_4\:.
\]
A general hyperplane section is a quintic elliptic curve in $\P^3$,
which is not projectively normal. In fact, the linear system of 
hyperplane sections is not complete, and therefore the curve is 
not a (linear) normal curve.
\end{example}

\subsection{The $\delta$-invariant}
Let $(X,0)\subset (\C^N,0)$ be an isolated singularity.
One  measures how far the singularity is
from being normal with the \textit{$\delta$-invariant}: 
\[
\delta(X,0)= \dim(\sier{ \wl X,\nu^{-1}(0)}/\sier
{X,0})\;.
\]
For plane curves this is the familiar $\delta$-invariant,
which is also called the number of virtual double points.
In higher dimensions it is not the correct double point number.
One has to consider the double point locus
of the composed map $\vp
 \colon \wl X\to X   \to\C^N$
(see \cite{Kl} for the general theory of double point schemes).  
The expected dimension of the double point locus is
$2 \dim X -N$.
As $X$ has  an isolated singularity, necessarily  
$N\geq 2\dim  X$, and the best results are in case of equality.

Consider a map $\vp \colon \wl X \to Y$ with $\wl X$
a complete intersection and $Y$ smooth of dimension twice the
dimension of
$\wl X$. 
As measure of degeneracy
of the map $\vp$ Artin and Nagata \cite{AN} introduced (following 
Mumford) half the size of source double point locus :
\[
\Delta(\vp) = \half12
\dim \Ker \left(\sier{\wl X\times_Y\wl X} \lra \sier X\right)\;.
\]
This dimension is stable under deformations of $\vp$, and by deforming
to an immersion with only nodes one sees that $\Delta(\vp)$ is 
an integer, as in that case $\wl X\times_Y\wl X$ splits into
the diagonal and a finite set with free $\Z/2$-action.

For plane curve singularities it follows by
deforming to a curve with only nodes  that 
$\delta(X)=
\Delta(\vp)$, cf.\ \cite[3.4]{Tei}.
As the image of $\vp$ is given by one equation, the images of a 
family of maps form a flat family and 
$\sier{\wl X_S}/\vp^*\sier {Y_S}$ is flat over the base $S$.
In higher dimensions the images of
a family of maps need not form a flat family and 
$\delta(X)$ may be larger than $\Delta(\vp)$.

\begin{example}[{cf.~\cite[(5.8)]{AN}}]
Map  three copies of $(\C^2,0)$ generically  to $(\C^4,0)$, say by
$(x,y,z,w)=(s_1,t_1,0,0)=(0,0,s_2,t_2)=(s_3,t_3,s_3,t_3)$.
The image of this map $\vp$ 
lies on the quadric $xw=yz$, and there are four more cubic
equations. The singularity is rigid, one computes that $T^1=0$.
The  $\delta$-invariant is  equal to $4$. 
On the other hand, the double point number $\Delta(\vp)$
is $3$, and the general deformation of the map
is obtained by moving the third
plane. The ideal of the image is then the intersection of
the ideals
$(z,w)$, $(x,y)$ and $(x-z-a,y-w-b)$. 
There are eight cubic equations,
obtained by multiplying the generators of the three ideals in all possible 
ways. Specialising to $a=b=0$ one obtains the product of the ideals 
$(z,w)$, $(x,y)$ and $(x-z,y-w)$. This ideal has an embedded 
component. The same ideal is obtained if one does not consider
the image with its reduced structure, but with its
Fitting ideal structure, as in \cite[\S 1]{Tei};
indeed, that construction commutes with base change.
\end{example}

\subsection{Simultaneous normalisation}

\begin{definition}
Let $f\colon \cX \to S$ be a flat map between complex spaces, such that
all  fibres are reduced. A \textit{simultaneous normalisation} of
$f$ is a finite map $\nu\colon \wl{\cX}\to \cX$ such that all fibres of the
composed map $f\circ \nu $ are normal, and that for each $s\in S$ the
induced map on the fibre $\nu_s\colon \wl{\cX}_s=(f\circ \nu)^{-1}(s)
\to \cX_s=f^{-1}(s)$ is the normalisation.
\end{definition}

Criteria for the existence of a simultaneous normalisation are given by 
Chiang-Hsieh and  Lipman \cite{CL}, see also \cite[II.2.6]{GLS}.
If the nonnormal locus of $\cX$ is finite over the base $S$, and
$S$ is smooth 1-dimensional, then the normalisation of $\cX$ is 
a simultaneous normalisation if and only if $\delta (\cX_s)$ (defined
as the sum over the $\delta$-invariants of the singular points)
is constant. For families of curves this results holds over 
an arbitrary normal base $S$, a result originally due to Teissier
and Raynaud.

\subsection{The geometric genus}
The geometric genus of a normal surface singularities was introduced
by Wagreich \cite{Wag}, 
using a resolution $\pi\colon (\wt X,E)\to (X,0)$,  
as  $p_g=\dim H^1(\wt X,\sier {\wt X})$. In
terms of cycles on the resolution one has $H^1(\wt X,\sier {\wt
X})={\displaystyle \lim_{\longleftarrow}}\;
 H^1(Z,\sier Z)$, where $Z$ runs over all effective divisors with
support on the exceptional set. This means that $p_g$ is the maximal value 
of $h^1(\sier Z)$, see \cite[4.8]{rpc}. 
Wag\-reich also defined the arithmetic genus
of the singularity as the maximal value of $p_a(Z)$, where
$p_a(Z)=1-h^0(\sier Z)+ h^1(\sier Z)$. This is a topological
invariant. 
One  computes $p_a(Z)$ by the adjunction formula: $p_a(Z)=1+\frac12 Z(Z+K)$.
The geometric genus has an interpretation
independent of a resolution, as $\dim
(H^0(U,\Omega_U^2)/L^2(U,\Omega_U^2))$, where $L^2(U,\Omega_U^2)$ is
the subspace of square-integrable
$2$-forms on $U=\wt X\setminus E=X\setminus 0$   \cite{Lau}. 

For a not necessarily normal isolated singularity $(X,0)$ 
the geometric genus is a combination of the $\delta$-invariant
and invariants from the resolution. This makes sense, as
the resolution factors over the normalisation.
In any dimension we define, following Steenbrink \cite[(2.12)]{Stee}:

\begin{definition}
Let $(X,0)$ be an isolated singularity of  pure dimension $n$
with resolution $(\wt X,E)$.
The \textit{geometric genus} is
\[
p_g(X,0)= -\delta(X,0)+\sum_{q=1}^{n-1}(-1)^{q-1}\dim H^q(\wt X,\sier {\wt
X})\;.
\]
\end{definition} 
The dimension of $H^q(\wt X,\sier {\wt X})$ does not depend
on the chosen resolution, and is therefore an invariant
of the singularity; one way to see this is using an
intrinsic characterisation: one has  $H^q(\wt X,\sier {\wt X})\cong
H^{q+1}_{\{0\}}(X,\sier X)$ for $1\leq q \leq n-2$, and 
 $H^{n-1}(\wt X,\sier {\wt X})\cong 
H^0(U,\Omega_U^n)/L^2(U,\Omega_U^n)$
\cite[Prop.~4.2]{Kar}. We remark that $\delta(X,0)=\dim
H^1_{\{0\}}(X,\sier X)$.  For isolated Cohen-Macaulay singularities all
terms except the last one vanish, so
$p_g=(-1)^n\dim H^{n-1}(\wt X,\sier {\wt X})$, which is the 
direct generalisation of Wagreich's formula.

By the results of \cite{Elk} the geometric genus is semicontinuous
under deformation.
More precisely, let $\pi\colon \wt X \to X$ be a resolution
of a pure dimensional space
and let the complex $M_X^\bullet$ be the third vertex of the triangle
constructed on the natural map 
$i\colon \sier X\to R^\bullet \pi_*\sier {\wt X}$:
\begin{eqnarray*} 
& \quad M_X^\bullet\\
 \makebox[0pt]{\raisebox{5pt}{$\scriptstyle +1$}$\!\!\swarrow$} &&\quad\makebox[0pt][r]{$\nwarrow$} \\
\sier X & \lra & R^\bullet \pi_*\sier {\wt X}
\end{eqnarray*}
Then $M_X^{-1}=\Ker \sier X \to \sier {X_{\text{red}}}$,
$M_X^0=\sier{\wl X }/\sier X$, $M_X^i=R^i\pi_*\sier{\wt X}$ for 
$0<i<n=\dim X$ and all other $M_X^i$ are zero.
If $X$ has isolated singularities define the partial Euler-Poincar\'e
characteristics
\[
\psi_i(X)=\sum_{j=0}^{n-i} (-1)^j \dim M_X^{n-j-i-1}, \qquad 0\leq i \leq n\;.
\]

\begin{proposition}[{\cite[Th\'eor\`eme 1]{Elk}}]
For an equidimensional flat morphism $f\colon \cX \to S$ with $\cX$
smooth outside a closed set, finite over the base,
the functions $s\mapsto \psi_i(\cX_s)$ are upper semicontinous.
\end{proposition}

This result has the following corollaries, which are relevant for us.

\begin{corollary}[{\cite[p.~255]{Mum}}]\label{cormum}
If a nonnormal reduced isolated surface singularity $X$ 
is smoothable, then $\delta(X)\leq p_g(\wt X)$.
\end{corollary}

\begin{corollary}[{\cite[(14.2)]{Kol}}]\label{corkol}
Let $f\colon X\to T$ be a morphism from a normal threefold to the germ
of a smooth curve. If $X_0=f^{-1}(0)$ has only isolated 
singularities  and the normalisation $\wl X_0$ has only rational
singularities, then $\wl X_0=X_0$. 
\end{corollary}

For a 1-parameter deformation of an  isolated nonnormal surface singularity
with rational normalisation semicontinuity of $\psi_0=-\delta $ and $\psi_1
=\delta$
implies that $\delta$ is constant. Therefore there is a simultaneous 
normalisation. 
The same is not necesarily true for infinitesimal deformations.
In the next section we give an example, where there exist obstructed
deformations without simultaneous normalisation.

\section{Computations} \label{sect_comp}
In this section 
we describe  equations and  
deformations for  surface singularities with $\delta=1$,
whose normalisation is a double point, so  given by an equation of the form
 $z^2=f(x,y)$, with
$f\in \m^2$. 

We recall the set-up for 
deformations of singularities (for details see \cite{knis}). 
One starts from a system of
generators $(g_1,\dots,g_k)$  of the ideal of the singularity $X$.
We also need generators of the module of relations, which  we
write as matrix $(r_{ij})$,
$i=1,\dots,k$, $j=1,\dots,l$. So we have $l$
relations $\sum g_ir_{ij}=0$. We perturb
the generators to
$G_i(x,t)$ with $G_i(x,0)=g_i(x)$. 
These describe a (flat) deformation
of $X$ if it is possible to lift the relations: there should exist a 
matrix $R(x,t)$ with $R(x,0)=r(x)$ such that $\sum G_iR_{ij}=0$ for 
all $j$. One can take this as definition of flatness.
In particular,
for an infinitesimal deformation $G_i(x,\ep)=g_i(x)+\ep
g_i'(x)$ (with $\ep^2=0$)
one needs the existence of a matrix $r'(x)$ such that
$\sum (g_i+\ep g_i')(r_{ij}+\ep r_{ij}')
= \ep \sum (g_i'r_{ij}+g_ir_{ij}')=0$, or equivalently that 
$\sum  g_i'r_{ij}$ lies in the ideal generated by the $g_i$.
Deformations, induced by coordinate transformations, are
considered to be trivial.
To find the versal deformation, one takes representatives for
all possible non-trivial infinitesimal deformations, and tries to lift to
higher order. The obstructions to do this define the base space
of the versal deformation.

We consider a subring $\cO$ of
$\wl\cO=\C\{x,y,z\}/(z^2-f(x,y))$ with $\delta=\dim \wl\cO/\cO=1$. 
We need a system of generators for the defining ideal. This, and
the possible deformations depend
on the subring in question. We write
$\cO=\C+L+\m^2$, where $L$ is a two-dimensional subspace of
$\m/\m^2$, which can be given as kernel of a linear form $l\colon
\m/\m^2 \to\C$, $l=ax+by+cz$.

First suppose that $c\neq 0$; we may assume that $c=1$. Generators
of $\cO$ are then
\begin{equation}\label{generat}
\begin{aligned}
\xi_1&=x-az &\qquad \eta_1&=y-bz  &\qquad \zeta_2&=z^2\\
\xi_2&=z(x-az) & \eta_2&=z(y-bz)  & \zeta_3&=z^3
\end{aligned}
\end{equation}
This system of generators is not minimal, as we not yet have taken
the relation $z^2=f(x,y)$ into account. Because $f\in\m^2$, it can be
written in terms of the generators; for example,
$x^2=\xi_1^2+2az\xi_1+a^2z^2=
\xi_1^2+2a\xi_2+a^2\zeta_2$, and
$x^3=\xi_1^3+3a\xi_1\xi_2+3a^2\zeta_2\xi_1+a^3\zeta_3$.
The relations $z^2=f(x,y)$ and $z^3=zf(x,y)$
lead to  equations
\begin{align*}
\zeta_2&=\varphi_2(\xi_1,\xi_2,\eta_1,\eta_2,\zeta_2,\zeta_3)\;,
\\
\zeta_3&=\varphi_3(\xi_1,\xi_2,\eta_1,\eta_2,\zeta_2,\zeta_3)\;.
\end{align*}
Therefore the variables $\zeta_2$ and $\zeta_3$ can be
eliminated and the embedding dimension of $\cO$ is four. Coordinates
are $\xi_1$, $\xi_2$, $\eta_1$ and $\eta_2$, and equations
for the corresponding singularity $X$  are
\begin{equation}\label{eq-betw-gen}
\begin{split}
    \xi_1\eta_2&=\xi_2\eta_1\\
    \xi_2^2&=\xi_1^2\varphi_2(\xi_1,\xi_2,\eta_1,\eta_2)\\
    \xi_2\eta_2&=\xi_1\eta_1\varphi_2(\xi_1,\xi_2,\eta_1,\eta_2)\\
    \eta_2^2&=\eta_1^2\varphi_2(\xi_1,\xi_2,\eta_1,\eta_2)
\end{split}
\end{equation} 
where $\varphi_2(\xi_1,\xi_2,\eta_1,\eta_2)$ is obtained from
$\varphi_2(\xi_1,\xi_2,\eta_1,\eta_2,\zeta_2,\zeta_3)$ by
eliminating $\zeta_2$ and $\zeta_3$.

\begin{proposition}
  The  nonnormal singularity with equations \eqref{eq-betw-gen}
has only deformations with simultaneous
  normalisation.
\end{proposition}

\begin{proof}
We write down the four relations between the generators:
\begin{align*}
    (\xi_2\eta_2-\xi_1\eta_1\varphi_2)\xi_1-(\xi_2^2-\xi_1^2\varphi_2)\eta_1
  -(\xi_1\eta_2-\xi_2\eta_1)\xi_2&=0\\
    (\eta_2^2-\eta_1^2\varphi_2)\xi_1-(\xi_2\eta_2-\xi_1\eta_1\varphi_2)\eta_1
  -(\xi_1\eta_2-\xi_2\eta_1)\eta_2&=0\\
  (\xi_2\eta_2-\xi_1\eta_1\varphi_2)\xi_2-(\xi_2^2-\xi_1^2\varphi_2)\eta_2
  -(\xi_1\eta_2-\xi_2\eta_1)\xi_1\varphi_2&=0 \\
   (\eta_2^2-\eta_1^2\varphi_2)\xi_2-(\xi_2\eta_2-\xi_1\eta_1\varphi_2)\eta_2  
   -(\xi_1\eta_2-\xi_2\eta_1)\eta_1\varphi_2&=0
\end{align*}
All perturbations of the equations lie in the maximal ideal. By
using coordinate transformations we can assume that the first
equation $\xi_1\eta_2-\xi_2\eta_1$ is not perturbed at all. 
We perturb the other equations as
$\xi_2^2-\xi_1^2\varphi_2-\ep_{\xi\xi}$,
$\xi_2\eta_2-\xi_1\eta_1\varphi_2-\ep_{\xi\eta}$ and
$\eta_2^2-\eta_1^2\varphi_2-\ep_{\eta\eta}$.
We get four equations holding in $\cO$, which can be written as
\[
\rk \begin{pmatrix} \xi_2&\xi_1&\ep_{\xi\xi}&\ep_{\xi\eta}\\
\eta_2&\eta_1&\ep_{\xi\eta}&\ep_{\eta\eta}\end{pmatrix} \leq1\;.
\]
The first minor is the equation $\xi_1\eta_2-\xi_2\eta_1$, and the
last minor vanishes identically, 
as we are considering infinitesimal deformations.
Thus  the perturbations can be written as
$\ep_{\xi\xi}=\xi_1^2\ep_{11}+\xi_1\xi_2\ep_{12}$ (without
$\xi_2^2$-term, as $\xi_2^2=\xi_1^2\varphi_2$),
$\ep_{\xi\eta}=\xi_1\eta_ 1\ep_{11}+\xi_1\eta_2\ep_{12}$ and
$\ep_{\eta\eta}=\eta_1^2\ep_{11}+\eta_1\eta_2\ep_{12}$ with 
$\ep_{11}$ and $\ep_{12}$ the same functions of the variables
in all three perturbations. We can arrange that $\ep_{11}$ only
depends on $\xi_1$ and $\eta_1$, and not on $\xi_2$ and $\eta_2$,
by collecting terms in $\ep_{12}$.
The coordinate transformation
$\xi_2\mapsto \xi_2+\frac12\xi_1\ep_{12}$,  
$\eta_2\mapsto \eta_2+\frac12\eta_1\ep_{12}$
gets rid of the terms with $\ep_ {12}$.
The resulting equations
\begin{align*}
    \xi_1\eta_2&=\xi_2\eta_1\\
    \xi_2^2&=\xi_1^2(\varphi_2+\ep_{11})\\
    \xi_2\eta_2&=\xi_1(\varphi_2+\ep_{11})\\
    \eta_2^2&=\eta_1^2(\varphi_2+\ep_{11})
\end{align*} 
define not only an
infinitesimal deformation, but also a genuine deformation. We conclude
that the base space of the versal deformation is smooth (even though
$T^2$ is not zero).

The simultaneous normalisation is given by $z^2=f+\ep_{11}(x-az,y-bz)$.
\end{proof}

To obtain interesting other deformations we have to assume that $c=0$. 
Then the
subspace $L$ of $\m/\m^2$ is given as kernel of a linear form
$l=ax+by$. Assuming $a=1$ we find $z$ and $y-bx$ as generators of
degree 1. As we have not yet specified the form of $f(x,y)$ we can
apply a coordinate transformation to achieve that $b=0$. 
So we take the linear form $l=x$. Generators of the ring $\cO$ are
now $z$, $y$, $w=zx$, $v=yx$, $x_2=x^2$ and $x_3=x^3$. If none of
these monomials occurs in $f(x,y)$, then the embedding dimension is
6.

The formulas 
\begin{equation}\label{newgener}
\begin{aligned}
x_2&=x^2 &\qquad y&=y &\qquad z&=z\\
x_3&=x^3 & v&=xy & w&=xz 
\end{aligned}
\end{equation}
define an injective map $\nu \colon\C^3\to \C^6$. 
Let $Y$ be the image, which is an isolated 3-dimensional singularity.
As it  is not even normal, its nine
equations cannot directly be given in determinantal format, but this 
is possible by allowing some redundancy.
Consider the maximal minors of the $2\times 6$ matrix
\begin{equation}
\label{detvgl}
\begin{pmatrix}
z & y & x_2 & w & v & x_3 \\
w & v & x_3 & zx_2 & yx_2 & x_2^2
\end{pmatrix} \;.
\end{equation}
There are three equations occurring twice, like $vw-zyx_2$, while the
last three are obtained by multiplying the first three by $x_2$. Therefore
we get nine generators of the ideal. 
The determinantal format gives relations between the generators, and a
computation with \textsc{Singular} \cite{GPS}
shows that there are no other relations. A further computation gives  
$\dim T^1_Y=1$. 
The 1-parameter deformation is
given by the maximal minors of
\begin{equation}
\label{afbvgl}
\begin{pmatrix}
z & y & x_2 & w & v & x_3 \\
w & v & x_3 & z(x_2+s) & y(x_2+s) & x_2(x_2+s)
\end{pmatrix} \;.
\end{equation}
It comes from deforming the map
$\nu$ to
\begin{equation}
\label{afb}
\begin{aligned}
x_2&=x^2-s &\qquad y&=y &\qquad z&=z\\
x_3&=x(x^2-s) & v&=xy & w&=xz 
\end{aligned}
\end{equation}
The singularity of the
general fibre is isomorphic to the one-point union of two 3-spaces
in 6-space. 

Now we restrict the map $\nu$ to a hypersurface $ \{z^2=f(x,y)\}
\subset \C^3$, with $z^2-f \in \nu^*\m_6$. 
We assume that $f\in(y^2,yx^2,x^4)$.
We get two additional 
equations by writing $z^2-f$ and $x(z^2-f)$ in the coordinates on
$\C^6$. We write
\begin{equation}\label{extravgl}
\begin{split}
z^2&=y \al+x_2 \beta\;, \\ 
zw&=v \al + x_3\beta \;. 
\end{split}
\end{equation}
The second equation is obtained from the first by \textit{rolling factors} 
using the matrix \eqref{detvgl}, i.e., replacing in each monomial
one occurrence of an entry of the upper row by the entry of the lower
row in the same column.
One can roll once more, to give an expression for $w^2$, but as we have
the equation $w^2=x_2z^2$, the resulting equation is just the $z^2$-equation
multiplied by $x_2$.

The singularity has a large component with simultaneous 
normalisation. For this just perturb $z^2-f$ with elements of 
$\nu^*\m_ 6$. 
That means that we can write the two additional equations,
using rolling factors. This works also for the deformation of $\nu(\C^3)$
given by the  equations \eqref{afbvgl} and the map \eqref{afb}.
But there is also another deformation direction. 

\begin{proposition}
\label{infdef}
The singularity $X$ with normalisation of the form
$z^2=f(x,y)$, where $f\in (y^2,yx^2,x^4)$,
 and local ring with generators \eqref{newgener}
has an infinitesimal deformation, not tangent to the
component with simultaneous normalisation. 
\end{proposition}

\begin{proof}
The existence is suggested by a \textsc{Singular} \cite{GPS} computation
in  examples. The result can be checked by hand. 

We first give the relations between the equations. 
We write $g_i$ for the  equations of $Y$, coming from the matrix 
\eqref{detvgl}, and $h_k$ for the two additional equations \eqref{extravgl}.
A relation has the form $\sum g_ir_{ij} +h_ks_{kj}=0$. We can pull it
back to $\C^3$ with the map $\nu$. It then reduces to 
$(z^2-f)(s_{1j}+xs_{2j})=0$.  Therefore we find six relations, generating
all relations with non-zero $s_{kj}$. They are found by reading the
matrix product
\[
\begin{pmatrix}
w & -z \\ v & -y \\ x_3 & -x_2 \\ zx_2& -w\\yx_2 &-v \\x_2^2 & -x_3
\end{pmatrix}
\begin{pmatrix}
z & y &x_2 \\ w &v & x_3
\end{pmatrix}
\begin{pmatrix}
z \\ -\al \\ -\beta
\end{pmatrix}
\]
in two different ways: the product  of the last two matrices
is a column vector containing the two equations $h_1$, $h_2$,
while the product of the first two is 
a $6\times 3$ matrix, with antisymmetric upper half
containing  the minors of the middle matrix (the
first half of the matrix \eqref{detvgl}) and symmetric lower half,
containing the remaining six generators. 
The other relations, with $s_{kj}=0$, are the determinantal
relations between the
equations $g_i$ of the three-dimensional singularity $Y$.

To find a
solution to $\sum g_i'r_{ij} + h_k's_{kj}=0\in \sier X$ it suffices 
to compute on 
the normalisation. 
We start with the determinantal relations for the $g_i$.
As the second row of the matrix is just the first
one multiplied with $x\in \wlsier X$, it suffices to consider only 
those relations 
obtained by doubling the first row. 
Consider the relation
\[
0=
\begin{vmatrix}
z & y & x_2 \\
z & y & x_2 \\
w & v & x_3
\end{vmatrix}
=
(yx_3-vx_2)z-(zx_3-wx_2)y+(zv-yw)x_2\;.
\]
The perturbation of the
three equations involved,
obtained from 
$z\cdot z - \al \cdot y -\beta \cdot x_2=0\in \wlsier X$,
cannot be extended to the other equations. It is possible to extend
after multiplication with $x\in \wlsier X$. Let $\wl\al\in \sier X$
be the element with $\wl\al= x\al \in\wlsier X$ and likewise
$\wl\beta=x\beta$. Then $zw=y\wl\al+x_2\wl\beta$.
We do not perturb the equations $h_1$ and $h_2$, nor the equation
$x_3^2-x_2^3$. We solve for the perturbations of the 
remaining equations, and check that all equations
 $\sum g_i'r_{ij} =0\in \sier X$ described above are
satisfied. 
The result is the following infinitesimal deformation 
(written as column vector): 
\[
G^t=g^t+\ep g'^t=
\begin{pmatrix}
zv-yw       \\
zx_3-wx_2    \\
yx_3-vx_2    \\
w^2-x_2z^2  \\
wv-x_2zy    \\
v^2-x_2y^2   \\
wx_3-zx_2^2  \\
vx_3-yx_2^2  \\
x_3^2-x_2^3  \end{pmatrix}
+\ep
\begin{pmatrix}
 \overline{\beta} \\
-\overline{\al} \\
-w \\
2z\al \\
  2y\al+x_2\beta  \\
  2zy \\
 x_2\al \\
 zx_2\\
 0
\end{pmatrix}
\]
\end{proof}

For the extension to higher order one needs 
further divisibility properties of 
$\al$ and $\beta$. Indeed, if $p_g(\wl X)=0$, then by
Corollaries \ref{cormum} and \ref{corkol}
the deformation of the Proposition has to be obstructed.

\begin{example}\label{rat_example}
Let the normalisation be a rational double point. 
Specifically, we take $\wl X$  of type $A_3$, given by $z^2=y^2+x^4$.  
For the nonnormal singularity $X$ the dimension of $T^1$ is equal to 8.
There is a 7-dimensional component with simultaneous normalisation: 
six parameters are seen in the equation
$z^2-y^2-x_2^2+a_1z+a_2y+a_3x_2+a_4w+a_5v+a_6x_3$, and $s$ is a
parameter for the deformation \eqref{afb} of the map  $\nu\colon \C^3\to\C^6$.
Finally let $t$ be the coordinate for infinitesimal deformation
of Proposition \ref{infdef}. A computation of the versal
deformation with \textsc{Singular} \cite{GPS, Ma} shows that 
the equations for the base space are 
$st=a_1t=a_2t=a_3t=a_4t^2=a_5t^2=a_6t^2=t^3=0$.
\end{example}

For $\wl X$ given by $z^2=f(x,y)$ with $f\in \mathfrak{m}^k$ the
structure of the versal deformation stabilises for large $k$.
Computations in examples with  \textsc{Singular} \cite{GPS} 
suggest that $T^2$ always has dimension  16 (this is also true 
for the singularity of Example \ref{rat_example})
and that there are in general 11 equations for the base space.
There is one component of codimension 1 with simultaneous
normalisation and two other components: a singularity $z^2=f(x,y)$
with $f\in \mathfrak{m}^k$, $k\ge6$  can be deformed into
$\wt E_7$ or $\wt E_8$.

We compute the component related to $\wt E_7$. This can be done by 
determining the versal deformation in negative degrees of the
singularity $z^2=-ay^4+bx^5$, where $a$ and $b$ are parameters,
using \cite{Ma}.
After a coordinate transformation the equations of the base space
do not depend on the parameters. We find the component. For other 
singularities we have just to substitute suitable functions
of space and deformation variables for the parameters in the 
formulas we find.

The result is rather complicated, so
we do not give all equations, but use 
\[
G_3 = x_3y-x_2v+tw
\]
to eliminate the variable $w$.
Four of the original equations do not involve $w$.
They are
\begin{align*}
G_6&= v(v+a_1t^2)-x_2y^2-2tzy-bt^2x_3+a_2t^2y^2+a_0t^2(x_2+a_2t^2)-a_3bt^4y\;,
\\
G_8&= x_3v-x_2^2y-tzx_2-2a_4t^2y(y^2+a_0t^2)\\
   &\qquad
    +ba_4t^4(v+a_1t^2) -(a_3v+a_2x_2)t^2y-a_2t^3z-a_3bt^4(x_2+a_2t^2)\;,
\\
G_9&=x_3^2-x_2(x_2+a_2t^2)^2-a_3t^2(x_2+a_2t^2)(v+a_1t^2)\\
&\qquad
   +a_4t^2(v+a_1t^2)^2-4a_4t^2x_2y^2-a_3^2t^4y^2\;,
\\
H_1&=z^2+a_4(y^2+a_0t^2)^2-bx_3x_2 \\
  &\qquad +(a_3v+a_2x_2)(y^2+a_0t^2)+a_1x_2v+a_0x_2^2
  +a_3bt^3z-a_4b^2t^4x_2\;.
\end{align*}
The ideal with $w$ eliminated has three more generators,
which we give the name of the original generators leading to them.
\begin{align*}
G_1&=x_3(y^2+a_0t^2)-x_2yv+zt(v+a_1t^2)\\
   &\qquad
-bt^2x_2(x_2+a_2t^2)
  +a_3t^2y(y^2+a_0t^2)+a_1t^2x_2y\;,\\
G_2& =x_3x_2y-x_2v(x_2+a_2t^2)+tzx_3-a_4t^2(v+a_1t^2)(y^2+a_0t^2)\\
&\qquad
  +2a_4bt^4x_2y+a_3a_2t^4(y^2+a_0t^2)-a_3t^3zy+a_3a_0t^4x_2\;,
\\
H_2&=x_2z(v+a_1t^2)-x_3yz+a_4t(v+a_1t^2)y(y^2+a_0t^2)\\
   &\qquad
+a_3tx_2y^3+a_2tx_2vy
 +ta_1x_2^2y+ta_0x_3x_2\\
&\qquad
  -btx_2^2(x_2+a_2t^2)+a_3t^2zy^2-2a_4bt^3x_2y^2-a_3a_2t^3y(y^2+a_0t^2)\;.
\end{align*}
To obtain the full ideal one has to add the equation used
to eliminate $w$ and saturate with respect to the variable $t$.

\begin{example}\label{eseven}
We use our equations to write down the deformation in the case that
$z^2-f$ is a surface singularity of type $\wt E_7$. 
We start from $z^2=y^4-\nu x^2y^2+x^4$. 
There is a second modulus, coming from
changing the generator $y$ to $y+\la x$. By a coordinate transformation
we can keep $y$ as generator, and take 
\begin{equation}\label{ellvgl}
z^2=y^4-\mu xy^3-\nu x^2y^2+x^4
\end{equation} 
as normalisation.
Fixing these moduli the component is one-dimensional. We describe
its total space.
Its equations are obtained by putting 
$b=a_1=0$, $a_0=a_4=-1$, $a_3=\mu$ and $a_2=\nu$ in the formulas above.
The equation $H_1$ becomes
\[
H_1=z^2-(y^2-t^2)^2-x_2^2+(\mu v+\nu x_2)(y^2-t^2)\;.
\]
It is reducible if $y^2=t^2$.
If $y^2 \neq t^2$, the equation $G_1$
shows that $x_3$ also can be eliminated.
There is one more equation not involving $x_3$:
\[
G_6=v^2-x_2(y^2+t^2)-2tzy+ \nu t^2(y^2-t^2)\;.
\]

The local ring of  the total space is a section ring $\oplus 
H^0(V,nL)$ for some ample line bundle on a projective surface $V$. 
The dimension of $H^0(V,L)$ is two.
We look at the
normalisation of a general hyperplane section $t=\la y$. 
We assume that $\la^2\neq1$,
Equation $G_6$ shows that $v/y$ is
in the normalisation. We set it equal to $(1-\la^2) x$.

If $\la^2+1\neq 0$ we can eliminate $x_2$ and find (after dividing
by $(1-\la^2)^2$) that 
that the normalisation
is given by 
\[
(z+2\la t^2-\nu\la y^2)^2=(\la^2+1)^2(y^4-\mu xy^3-\nu x^2y^2+x^4)\;,
\]
which for all $\la$ (with $\la^2+1\neq0$) is isomorphic to \eqref{ellvgl}.
The sections with $\la^2=1$ are reducible.
One sees that $V$ is a ruled surface over the elliptic curve
with equation \eqref{ellvgl}, with two sections of self-intersection
zero, $E_1$ and $E_{-1}$, and $L$ is given by the linear system
$|E_1+f|$, where $f$ is a fibre. The  general element of 
the linear system is a section
of the ruled surface with self intersection $2$.
\end{example}

\begin{remark}
Each nonnormal singularity $X\subset \C^N$  is the image of its 
normalisation $\wl X$, giving rise to a map $\wl X \to \C^N$.
Not every deformation of this map is flat. To give an example for
$X$ as above with normalisation $z^2=f(x,y)$, 
we observe that we can deform the map
by using the same map $\C^3\to\C^6$ and perturbing the equation
arbitrarily, say $z^2=f(x,y)+u$. For flatness of the images
one needs to perturb both equations 
$z^2=y\al + x_2\beta$, $zw=v\al+x_3\beta$
with elements in the local ring of the nonnormal singularity, where the
second is obtained from the first by multiplying with $x$ 
(on the normalisation). The perturbation $z^2=f(x,y)+u$ is not
of this type.
\end{remark}

\section{Deformations of a resolution}
A deformation of a resolution of a normal surface
singularity  blows down to a deformation of the singularity
if and only if $h^1(\sier{\wt X})$ is constant \cite{Rie, Wah}. If not, the total
space of a 1-parameter deformation still blows down to a
three-dimensional singularity, but the special fibre is no longer
normal.

Let more generally
$\pi_0\colon (Y,E)\to(X,0)$ be the contraction of an
exceptional set $E$ to a point, with $(Y,E)$ not
necessarily smooth, of dimension $n$. In principle $Y$ is a germ along $E$,
but we work always with a strictly pseudo-convex 
representative, which we denote with the same symbol $Y$. 
Then $\sier X=(\pi_0)_*\sier Y$. In
particular, $X$ is normal if $Y$ is normal. Consider now a
deformation $\tilde f\colon \cY \to S$ of $Y$ over a reduced base
space $(S,0)$. One can assume that $\tilde f$ has a 1-convex
representative. All the exceptional sets in all fibres can be
contracted: let $\pi \colon \cY \to \cX$ be the Remmert
reduction, so $\sier {\cX}=\pi_*\sier{\cY}$ with
$\tilde f=f\circ\pi$. Then $f \colon \cX\to S$ is a
deformation of $\cX_0:=f^{-1}(0)$. The question is whether $f$
also is a deformation of $X$, i.e., whether $X\cong \cX_0$. The
answer is the following \cite[Satz 3]{Rie}, cf.~\cite{Wah} for the
algebraic case.

\begin{theorem}
Let $\cY \mapright \pi \cX \mapright f S$ be the Remmert
reduction of the deformation $\tilde  f\colon \cY \to S$ of $Y$,
over a reduced base space $(S,0)$. Then the special fibre $\cX_0$ of
$f \colon \cX\to S$ is the Remmert reduction of $Y=\cY_0 $ if and
only if the restriction map $H^0(Y,\sier \cY)\to H^0(Y,\sier Y )$ is
surjective. This is the case if $\dim H^1(\cY_s,\sier {\cY_s})$ is
constant on $(S,0)$.
\end{theorem}

The converse of the last clause holds if $\dim H^2(\cY_s,\sier
{\cY_s})$ is constant \cite[Satz 5]{Rie}. This is automatically
satisfied in the case of interest to us, that $Y\to X$ is a
modification of a normal surface singularity. The result then 
says that a deformation of the modification blows down to a
deformation of the singularity if and only if $p_g$ is constant.

In the proof one reduces to the case of a 1-parameter deformation.
Let us consider what happens in that case, so we have a diagram
\[
\cdmatrix{\cY &
\smash{\mathop{\relbar\joinrel\longrightarrow}\limits^{\pi}}
 & \cX \\
    \makebox[0pt][l]{$\;\searrow$} &&\makebox[0pt][r]{$\swarrow\;$} \\
    &T}
\]
Let $t$ be the coordinate on $T$, and consider multiplication with
$t$ on $\sier\cY$ and $\sier\cX$. We get the commutative diagram
\[
\def\quad{}
\begin{array}{c@{}c@{}c@{}c@{}c@{}c@{}c@{}c@{}c}
 0 \to{} & H^0(\sier{\cY})&{}\mapright{\cdot t}{} &
H^0(\sier {\cY})&{}\lra {}&H^0(\sier {Y})&{}\lra{}&
 H^1(\sier {\cY})
&{}\mapright {\cdot t}
H^1(\sier {\cY})
\\
&\mapdown \cong && \mapdown \cong &&\mapdown {} \\
0 \to & \sier {\cX}&\mapright {\cdot t} & \sier {\cX}&\lra &\sier
{\cX_0} &\lra&0
\end{array}
\]
It shows that $\sier {\cX_0}$ is equal to $H^0(\sier{Y})=\sier X$
if and only if the restriction map $H^0(\sier {\cY})\to H^0(\sier
{\cY})$ is surjective. If  $\dim H^1(\cY_t,\sier {\cY_t})$ is
constant, then  $H^1(\cY,\sier {\cY})$ is a free
$\sier{\cY}$-module, on which multiplication with $t$ is injective,
so the restriction map $H^0(\sier {\cY})\to H^0(\sier {Y})$ is
surjective.

To study the converse, and what happens if $\dim H^1(\cY_t,\sier
{\cY_t})$ is not constant, we restrict to the case that  $H^2(\sier {\cY})=0$.
\begin{proposition}\label{remmert}
Let $\cY \mapright \pi \cX \mapright f T$ be the Remmert
reduction of the 1-parameter deformation $\tilde f\colon \cY \to T$
of the space $Y$, with $H^2(\sier {\cY})=0$. Then
\[
\dim H^1(\cY_0,\sier {\cY_0})=\dim H^1(\cY_t,\sier {\cY_t})-\dim
\bigl((\pi_0)_*\sier Y/\sier {\cX_0}\bigr)\;,
\]
where $t\neq0$.
\end{proposition}
\begin{proof}
The upper line in the commutative diagram extends as 
\[
\begin{split}
   0 \to{}  H^0(\sier{\cY})\mapright{\cdot t}
H^0(\sier {\cY})&\lra  H^0(\sier {Y})\lra{} \\
&
     {}\lra   H^1(\sier {\cY})
\mapright {\cdot t} H^1(\sier {\cY})\lra H^1(\sier {Y})\;.
\end{split}
\]
The generic rank of the $\sier T$-module $H^1(\sier {\cY})$ is  
equal to $\dim H^1(\cY_t,\sier {\cY_t})$, 
but also equal to $\dim \Coker(\cdot t) - \dim \Ker(\cdot t)$. This
proves the formula.
\end{proof}

In particular, if $Y$ is a resolution of the normal surface
singularity $X$, 
then the formula of the proposition says that 
$p_g( {\cX_t})=p_g(X)-\delta=p_g({\cX_0})$. 
Here we use essentially that we have a 1-parameter
deformation: over a higher dimensional base space $\delta$ will be
larger than $p_g(X)-p_g( {\cX_t})$: for a 1-parameter
curve in the base $p_g(X)-p_g( {\cX_t})$ gives how many
functions fail to extend, but it will depend on the curve
which ones do not extend.

\begin{remark}
If we have a smoothing, or more generally if $p_g(\cX_t)=0$, 
then $H^1(\sier{\cY})$ is $t$-torsion and isomorphic to $H^1(\sier Y)$.
\end{remark}

The above Proposition gives a way to construct normal surface singularities
with the same resolution graph as a given hypersurface singularity, 
but with lower $p_g$. If the drop in $p_g$ is equal to $\delta$, we 
start from a nonnormal model of the hypersurface singularity with
$\delta$-invariant $\delta$ and compute deformations without 
simultaneous resolution. If the singularity is quasi-homogeneous,
deformations of positive weight will have constant topological
type of the resolution.

\begin{example}\label{ex_endrie}
Consider the hypersurface singularity $z^2=x^7+y^7$ with
resolution graph:
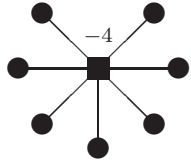
\[
\unitlength=30pt
\def\ci{\circle*{0.26}}
\begin{picture}(2,2)(-1,-1)
\put(0,0){\makebox(0,0){\rule{8pt}{8pt}}}
\put(0,.3){\makebox(0,0)[b]{$\scriptstyle -4$}}
\put(0,0){\line(1,0){1}} \put(1,0){\ci}
\put(0,0){\line(-1,0){1}} \put(-1,0){\ci}
\put(0,0){\line(0,-1){1}} \put(0,-1){\ci}
\put(0,0){\line(1,1){0.7}} \put(0.7,0.7){\ci}
\put(0,0){\line(-1,1){0.7}} \put(-0.7,0.7){\ci}
\put(0,0){\line(1,-1){0.7}} \put(0.7,-0.7){\ci}
\put(0,0){\line(-1,-1){0.7}} \put(-0.7,-0.7){\ci}
\end{picture}
\]
This is a singularity with $p_g=3$, but its arithmetic
genus  is equal to two. The general, non-Gorenstein
singularity with the same graph has indeed $p_g=2$.
We can use the computations of the previous section.
We get a weighted homogeneous deformation by putting 
$b=x_2$, $a_0=a_1=a_2=a_3=0$ and $a_4=-y^3$.
The equation $G_3=x_3y-x_2v+tw$ shows that the $w$-deformation
has positive weight $-(8-9)=1$.

We  give the equations for the fibre at $t=1$. Then the 
equation $H_2$ lies in the ideal of the other ones
and we obtain  the following six equations

\begin{align*}
G_6&= v^2-x_2y^2-2zy-x_2x_3\;,
\\
G_8&= x_3v-x_2^2y-zx_2+2y^6-x_2y^3v \;,
\\
G_9&=x_3^2-x_2^3-v^2y^3+4x_2y^5\;,
\\
H_1&=z^2-y^7-x_2^2x_3+y^3x_2^3\;.
\\
G_1&=x_3y^2-x_2yv+zv-x_2^3\;,
\\
G_2& =x_3x_2y-x_2^2v+zx_3+vy^5-2x_2^2y^4\;.
\end{align*}
One checks that this ideal indeed defines a singularity with the
above resolution graph by resolving it; one possible metod is to blow up
a canonical ideal. 
\end{example}

\section{$\P^1$ as exceptional set}
In understanding normal surface singularities the resolution is 
a very important tool. For 3-fold singularities this is not the
case for several reasons. First of all, there is no unique minimal 
resolution. The combinatorics of a good resolution (i.e., the
exceptional divisor has normal crossings) seems prohibitive
in general. But 
now there is a new phenomenon,
that there may exist resolutions in which the exceptional set is not
a divisor, but an analytic set of lower dimension. This
is called a small resolution. It means that in a certain
sense the singularity is not too singular. For 3-fold singularities
we are talking about resolutions with exceptional set a curve.

If the exceptional curve $C$ is rational, then its normal bundle
splits as $\cO(a)\oplus \cO(b)$ with $a\geq b$.  Rather surprisingly,
the number $a$ can be positive. 
Laufer gave in \cite{Lau2} an example 
of a   curve $C\subset \wt X$ with 
normal bundle
$\cO(1)\oplus\cO(-3)$, which even contracts to a hypersurface
singularity $X$. 

Generalising earlier results of 
Ando (see \cite{And1}) and Nakayama\cite{Nak}, that  
contractabi\-lity limits the value of $(a,b)$ to $2a+b<0$, 
Ando proved \cite{And2}:

\begin{theorem}
Let $C$ be a smooth exceptional curve in an $m$-dimensional
manifold $\wt X$, and let $M$ be a
subbundle of the normal bundle $N_{C/\wt X}$ of maximal degree
$a$ and put $b=\deg N_{C/\wt X} -a$. Then $2a+b<0$ and $a+b<0$.
Moreover, if $C$ is rational, then $a+b\leq 1-m$. 
\end{theorem}

Ando \cite{And,And2} has also existence results. In particular,
in dimension 3 he exhibits examples 
with the maximal normal bundle
$\cO(n)\oplus\cO(-2n-1)$, by
giving, in the style of Laufer,
transition functions between two copies of $\C^3$.
The resulting singularity is not  Cohen-Macaulay for $n>1$.
Consider more generally a rational curve with normal bundle ot type $(a,b)$
with $a>1$.
To see that $H^2_{\{0\}}(X,\sier X)\cong H^1(\wt X,\sier{\wt X})\neq 0$, 
let $\cJ$ be the ideal sheaf of $C$ in $\wt X$ and look at
the exact sequences
\begin{gather*}
0 \lra \cJ \lra \sier{\wt X} \lra \sier C\lra 0\;,\\
0\lra \cJ^2 \lra \cJ \lra \cJ/\cJ^2 \lra 0\;.
\end{gather*}
We have a surjection $H^0(\sier{\wt X}) 
\twoheadrightarrow H^0(\sier C)=\C$.
As $C$ is rational, $H^1(\sier C)=0$ and therefore
$H^1(\wt X,\sier{\wt X})\cong  H^1(\wt X, \cJ) $.
Because $C$ is a curve,  $H^2(\wt X,\cJ^2)=0$ and
we get a surjection $H^1(\wt X, \cJ) \twoheadrightarrow
H^1(C,\cJ/\cJ^2)$. As $\cJ/\cJ^2$ 
is the dual of the normal 
bundle, we have that $h^1(C,\cJ/\cJ^2)=a-1$ and therefore
$h^1(\sier{\wt X})\geq a-1$.

Pinkham gave a construction for $C$ with normal bundle
$\cO(1)\oplus\cO(-3)$ \cite{Pin2},
using smoothings of partial resolutions of rational double points.
The easiest example, starting from a $D_4$-singularity, is decribed
in detail in \cite{Ste-ne}.

Here we generalise Pinkham's construction to exceptional 
curves with other normal bundles. 
%
Let $\wl H$ be a normal surface singularity (in the end $H$
will be a general hyperplane section of a 3-fold singularity)
and let  $\wh H$ be a partial resolution of  $\wl H$ 
with irreducible exceptional
locus $C$, such that the only singularities of $\wh H$
are hypersurface singularities.

The deformation space of $\wh H$ is smooth. Indeed, the sheaf
$\cT^1_{\wh H}$ is concentrated in the singular points, and
$\cT^2_{\wh H}=0$, as there are only hypersurface singularities.
The local-to-global spectral sequence for $T_{\wh H}^\bullet$
gives that $T_{\wh H}^2=0$, so all deformations are unobstructed,
and moreover we get the exact sequence
\[
0 \lra H^1(\cT^1_{\wh H}) \lra T_{\wh H}^2 \lra \bigoplus_p T^1_{\wh H,p}
\lra 0
\;.
\]
Therefore all singular points $p\in \wh H$ can be smoothed
independently.

We take $\wt X$ to be a  1-parameter smoothing of $\wh H$
with smooth total space (this is possible, as all singularities
are hypersurfaces). Moreover, we can arrange that  the general fibre
does not contain exceptional curves.
Then the contraction $\pi\colon \wt X \to X$
with the curve $C$ as exceptional locus gives
an isolated 3-fold singularity. In general its hyperplane section
$H$ is a nonnormal surface singularity, with $\delta(H)=
p_g(\wl H)$, by Proposition \ref{remmert}.

\begin{example}\label{ex_xtnme}
Let $\wl H$ be  a singularity, whose
resolution has  a central rational
curve of self-intersection $-n-1$, intersected by $2n+1$  
$(-2)$-curves.
A quasi-homogeneous singularity with this resolution is the
hypersurface singularity $Y_{2n+1}$ with equation  $z^2=f_{2n+1}(x,y)$,
where $f_{2n+1}$ is a square-free binary form. 
The partial 
resolution $\wh H$ to be considered is obtained by blowing
down the $(-2)$-curves, intersecting the central curve. 
For $n>1$ this is the canonical model,
while for $n=1$ we have $D_4$.
The next Theorem shows that the 
construction yields a 3-dimensional manifold $\wt X$
with an exceptional rational curve, whose normal bundle is
$\cO(n)\oplus\cO(-2n-1)$.
\end{example}

Rational double points are absolutely isolated, i.e., they can be resolved
by blowing up points. 
Each sequence of blowing ups gives a partial resolution. We define the
\textit{resolution depth} of an exceptional component $E_i$
as the minimal number of blow ups required to obtain a partial resolution
on which the curve $E_i$ appears. 
This is the desingularisation depth of \cite{LT}, shifted by one.
It is easily computed from the resolution graph.
The fact that $C$ is smooth restricts the possible curves $E_i$
in a rational double points configurations, which intersect $C$, to those
with multiplicity one in the fundamental cycle of the configuration.

\begin{theorem}
Let $\wt X$ be a 1-parameter smoothing with smooth
total space
of a partial resolution $\wh H$
of a normal surface singularity $\wt H$ with exceptional
set a smooth rational curve $C$ and $k$ rational double points as 
singularities. Let $-c$ be the self-intersection of the curve $C$
on the minimal resolution $\wt H$ of $\wl H$. Suppose that
$C$ intersects a curve of resolution depth $b_j$ in the 
$j$th rational double point configuration on $\wt H$. Put
$b=b_1+\dots +b_k$.
Then the normal bundle
of the exceptional curve $C\subset \wh H$ in $\wt X$ is
$\cO(b-c)\oplus\cO(-b)$.
\end{theorem}

\begin{proof}
Let  $\sigma\colon \wt Y\to \wt X$ be an embedded resolution 
of $\wh H$. Denote by $\wt H$ the strict transform of $\wh H$ 
and by $\wt C$ the strict transform of the curve $C$ (which is isomorphic
to $C$). As we are only
interested in a neighbourhood of $\wt C$, it actually suffices to
blow up the threefold $\wt X$ in points lying on $C$ until the
strict transform of $\wh H$ is smooth along the strict transform of $C$.
The number of blow ups needed is $b$. 

Let $P_j\in C$ be the $j$th singular
point of $\wh H$; identifying $\wt C$ with $C$ it is also
the intersection point on $\wt H$ of $\wt C$ and 
the $j$th rational double point configuration.
The normal bundle $N_{\wt C/\wt Y}$ is isomorphic to
$N_{C/\wt X}\otimes \sier{C}(-D)$,
where we write $D$ for the divisor $\sum b_j P_j$.

On $\wt Y$ we have the exact sequence
\[ 
0\lra N_{\wt C/\wt H}  \lra  N_{\wt C/\wt y} \lra 
N_{\wt H/\wt Y}|_{\wt C} \lra 0\;.
\]
Correspondingly there is an exact sequence on $\wt X$:
\begin{equation}\label{exseq}
0\lra N' \lra  N_{C/\wt X} \lra N'' \lra 0\;,
\end{equation}
with $N'\cong N_{\wt C/\wt H} \otimes \sier{C}(D)$
a bundle, which outside the singular points
coincides with the normal bundle $N_{C/\wt H}$;
note that $C$ is not a Cartier divisor in $\wt H$ at the singular 
points.

As $\wt C$ is a rational curve, $N_{\wt C/\wt H} \cong \sier{\wt C}(-c)$.
To compute $N_{\wt H/\wt Y}|_{\wt C}$ we note that the
total transform of $\wh H$ is of the form $\wt H + \sum f_iF_i$, with 
$F_i$ the exceptional divisors, and that it is the divisor $t=0$
with trivial normal bundle. 
The only divisors, which intersect $\wt C$, are the ones coming from
the last blow up in the points $P_j$, and occur with multiplicity
$2b_j$.
Therefore
$N_{\wt H/\wt Y}|_{\wt C} \cong \sier{\wt C}(-\sum 2f_jF_j)=
\sier{\wt C}(-2b)$.
It follows that the
exact sequence \eqref{exseq} has the form
\[
0 \lra \cO(b-c) \lra N_{C/\wt X}\lra \cO(-b) \lra 0.
\]
As $H^0(N_{C/\wt X})=H^0(\cO(b-c))$, the sequence splits.
\end{proof}

\begin{example}
As noted by Ando \cite{And2}, there exist exceptional
rational curves with normal bundle $(1,-4)$, which contract to
Cohen-Macaulay singularities, and others which do not.
We obtain this normal bundle starting from a RDP resolution with a
central rational curve of self-intersection $-3$ (on the minimal
resolution) with four $A_i$ singularities on it.
Consider the following two resolution graphs:
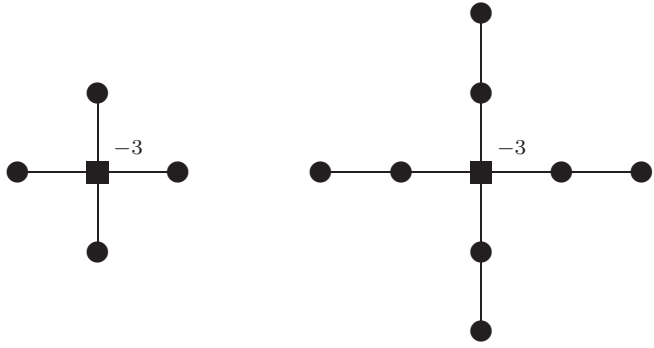
\[
\unitlength=30pt
\def\vir{\makebox(0,0){\rule{8pt}{8pt}}}
\def\ci{\circle*{0.26}}
\begin{picture}(3,3)(-1,-2)
\put(0,0){\vir}
\put(0.2,0.2){\makebox(0,0)[bl]{$\scriptstyle -3$}}
\put(0,0){\line(1,0){1}} \put(1,0){\ci}
\put(0,0){\line(-1,0){1}} \put(-1,0){\ci}
\put(0,0){\line(0,-1){1}} \put(0,-1){\ci}
\put(0,0){\line(0,1){1}} \put(0,1){\ci}
\end{picture}
\qquad
\begin{picture}(4,4)(-2,-2)
\put(0,0){\vir}
\put(0.2,0.2){\makebox(0,0)[bl]{$\scriptstyle -3$}}
\put(0,0){\line(1,0){1}} \put(1,0){\ci}
\put(1,0){\line(1,0){1}} \put(2,0){\ci}
\put(0,0){\line(-1,0){1}} \put(-1,0){\ci}
\put(-1,0){\line(-1,0){1}} \put(-2,0){\ci}
\put(0,0){\line(0,-1){1}} \put(0,-1){\ci}
\put(0,-1){\line(0,-1){1}} \put(0,-2){\ci}
\put(0,0){\line(0,1){1}} \put(0,1){\ci}
\put(0,1){\line(0,1){1}} \put(0,2){\ci}
\end{picture}
\]
The graph on the left is a rational quadruple point graph, 
while a normal singularity
with the second graph is minimally elliptic and equisingular to
$z^3=x^4+y^4$. The exceptional  $(1,-4)$-curve comes from a nonnormal
model with $\delta=1$. 
\end{example}

\begin{example}[Example \ref{ex_xtnme} continued]
Ando's examples \cite{And,And2}  of the extremal case $(n,-2n-1)$
are of type in the Example. 
With adapted variable names  his exceptional $\P^1$
is covered by two charts having
coordinates $(x,\eta,\zx)$  and  $(\xi,y,\zy)$ with
transition functions
\[
\begin{aligned}
  x  & =  \xi^{2n+1} y +\zy^2+\xi^{2n}\zy^3 \\
  \eta & = \xi^{-1}  \\
  \zx & =  \zy\xi^{-n} 
\end{aligned}
\]
So $x$ is a global function, as is $\xi y+\zy^3=\eta^{2n}x-\zx^2$. 
Other functions are more complicated.
A general hyperplane section is obtained by setting a linear combination
of these two functions to zero.
In the first chart we get $\zx^2=x(a+\eta^{2n})$ and in the second
$\zy^2+(1/a+ \xi^{2n})(\xi y+\zy^3)=0$. This is indeed the canonical
model of a singularity of 
type $z^2=f_{2n+1}(x,y)$.
\end{example}

\begin{remark}
For $n=2$ we have the singularity $z^2=f_5$, which is minimally elliptic
and therefore every singularity with the same resolution graph
is a double point. For $n>2$ this is no longer true. For $n=3$
we gave in Example \ref{ex_endrie} equations for a singularity
with the same resolution graph as $z^2=f_7$, 
which is not Gorenstein. As we constructed
it as deformation of a nonnormal model, with $\delta=1$,
of a hypersurface singularity, a nonnormal $\delta=2$ model of this
singularity is a deforation of a $\delta=3$ model of the 
hypersurface.
\end{remark}

Before giving other new examples with  normal bundle $(n,-2n-1)$
we recall Koll\'ar's length invariant \cite[Lecture 16]{CKM}:
\begin{definition}
The length $l$ of the small contraction $\pi : (\wt X,C) \to
(X, p)$  with irreducible exceptional curve C is
\[
l = \lg \sier{\wt X}/\pi^*\mathfrak{m}_{X,p}\;.
\]
\end{definition}
The length is equal to the multiplicity of the 
maximal ideal cycle of $\wl H$ at the
strict transform of the exceptional curve $H$.

\begin{example}
Consider the graph:
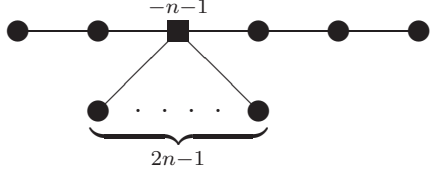
\[
\unitlength=30pt
\def\vir{\makebox(0,0){\rule{8pt}{8pt}}}
\def\ci{\circle*{0.26}}
\begin{picture}(5,2)(-2,-1.75)
\put(0,0){\vir}
\put(0,0.2){\makebox(0,0)[b]{$\scriptstyle -n-1$}}
\put(0,0){\line(1,0){1}} \put(1,0){\ci}
\put(1,0){\line(1,0){1}} \put(2,0){\ci}
\put(2,0){\line(1,0){1}} \put(3,0){\ci}
\put(0,0){\line(-1,0){1}} \put(-1,0){\ci}
\put(-1,0){\line(-1,0){1}} \put(-2,0){\ci}
\put(0,0){\line(-1,-1){1}} 
\put(0,0){\line(1,-1){1}} 
\put(-1.1,-1.2){$\underbrace{
\begin{picture}(2.2,1)(-0.1,-0.2)
\put(0,0){\ci}\put(1,0){\makebox(0,0){$.\;\;.\;\;.\;\;.$}}
\put(2,0) {\ci}
\end{picture}
}_{2n-1}$}
\end{picture}
\]
An example of a normal surface singularity with this graph 
is 
\[
z^2=y(y^{4n-2}+x^{6n-3})\;.
\]
These singularities can be thought of as generalisations of
$E_7$, just as those of the type  $z^2=f_{2n+1}(x,y)$
are generalisations of $D_4$.
A  smooth total space of a 1-parameter smoothing of the canonical model
has as exceptional set a rational curve with normal bundle $(n,-2n-1)$.
The invariant $l$ has the value $4$.

For $n=2$ the singularity has $p_g=5$.
\end{example}

From our computations in Section \ref{sect_comp}
we can draw the following conclusion.

\begin{proposition}
The singularity obtained by contracting a rational curve with normal 
bundle $(2,-5)$ has embedding dimension at least 7.
\end{proposition}

Explicit equations for the case that the normalisation $\wl H$ is
given by $z^2=y^5+x^5$, can be obtained
from the equations $G_i$, $H_j$ of Section \ref{sect_comp} by putting $b=1$,
$a_4=a_4-y$, $a_3=a_3+b_4y$, $a_2=a_2+b_3y$, $a_1=a_1+b_2y+c_2y^2$
and $a_0=a_0$.

The equation
\begin{multline*}
H_1=z^2-(y-a_4)(y^2+a_0t^2)^2-x_3x_2 
    +((a_3+b_4y)v+(a_2+b_3y)x_2)(y^2+a_0t^2) \\
  \qquad+(a_1+b_2y+c_3y^2)x_2v
+a_0x_2^2 +(a_3+b_4y)t^3z+(y-a_4)t^4x_2\;.
\end{multline*}
shows that restricted to $t=0$ one has the simultaneous
normalisation
\begin{multline*}
z^2-y^5-x^5\\+a_4y^4+a_3xy^3+a_2x^2y^2+a_1x^3y+a_0x^4 
    +b_4xy^4+b_3x^2y^3+b_2x^3y^2+c_3x^3y^3
\;.
\end{multline*}

To compute the simultaneous canonical model we first eliminate
$w$ and $x_3$, assuming  $t\neq0$. To simplify the formulas
we suppress the $b_i$ and $c_3$. We write the  
resulting  equations in  determinantal form: 
\begin{equation}\label{determatrix}
\begin{pmatrix}
v(y^2+a_0t^2)-x_2yt^2-zt^3-a_3t^6 &
v^2-(y-a_4)t^6\\
-(y^2+a_0t^2)^2+x_2t^4&
-v(y^2+a_0t^2)-x_2yt^2-zt^3\\
2y(y^2+a_0t^2)-vt^2-a_1t^4&
2vy+x_2t^2+a_2t^4
\end{pmatrix}\;.
\end{equation}
We apply a Tjurina modification
followed by normalisation. On the first chart  $\cU_y$
we have the hypersurface 
\[
\zy^2-y\xi^5-y+a_0\xi^4+a_1\xi^3+a_2\xi^2+a_3\xi+a_4-\half14t^2\xi^6
\]
and the map to the singularity $H$ is given by quite complicated
formulas:
\begin{align*}
x_2&=y^2\xi^2-2yt(\zy-t\xi^3)+
-t^3\xi(\zy-\half12t\xi^3)-a_0t^2\xi^2-a_1t^2\xi-a_2t^2
\\
v&=(y^2+a_0t^2)\xi-yt^2\xi^2+t^3(\zy-\half12t\xi^3)
\\
z&=y^2(\zy-\half52t\xi^3)+yt^2\xi(3\zy-\half52t\xi^3)+
t^4\xi^2(\zy-\half12t\xi^3)\\&{}
 \qquad-a_0t^2(\zy-\half32t\xi^3)+2ya_0t\xi^2
+ya_1t\xi+a_1t^3\xi^2+ya_2t+a_2t^3\xi
\end{align*}
The expressions for $w$ and $x_3$ are even longer; they can be computed
from the equations used to eliminate these variables.

In the second chart $\cU_x$ we have the hypersurface 
\[
\zx^2-x-x\eta^5+a_0+a_1\eta+a_2\eta^2+a_3\eta^3+a_4\eta^4-\half14t^2\eta^8
\]
and the transition functions
$y=x\eta+t\zx+\frac12t^2\eta^4$, $\xi=\eta^{-1}$ and 
$\zy=\zx\eta^{-2}+\frac12t\eta^{-3}+\frac12t\eta^2$.

We see that for $a_i=0$ the family of curves given in $\cU_y$ by
$y=\zy-\frac12t\xi^3=0$ and in $\cU_x$ by $x=\zx+\frac12t\eta^4=0$ is
exceptional. Then the determinant \eqref{determatrix} 
describes for $t\neq0$ a singularity
isomorphic to the
cone over the rational normal curve of degree three.
The general 1-parameter smoothing of the canonical model is
obtained by taking the $a_i$ as functions of $t$. We can take
$a_0=t$, $a_i=0$ for $i>0$. Then in the second chart $\tau=t(1-\half14t\eta^8)$
can be eliminated, as $\tau=x(1+\eta^5)-\zy^2$, so $(x,\eta,\zy)$ are
coordinates. We can write the equation in the first chart
as $\zy^2+\xi^4(t-\half14t\xi^2-y\xi)=y$, so 
$(\zy,\xi,\sigma=t-\half14t\xi^2-y\xi)$ are coordinates. The transition 
functions are power series.

\end{document}